\documentclass[11pt]{amsart}

\usepackage[usenames,dvipsnames]{pstricks} 
\usepackage{epsfig}
\usepackage{graphicx,color}
\usepackage{geometry}
\usepackage[all]{xy}
\usepackage{amssymb}
\usepackage{cite}
\usepackage{fullpage}
\xyoption{poly}

\newtheorem{introtheorem}{Theorem}
\newtheorem{introdefinition}[introtheorem]{Definition}

\swapnumbers
\newtheorem{theorem}{Theorem}[section]
\newtheorem{lemma}[theorem]{Lemma}
\newtheorem{proposition}[theorem]{Proposition}
\newtheorem{corollary}[theorem]{Corollary}
\theoremstyle{definition}
\newtheorem{definition}[theorem]{Definition}
\newtheorem{defprop}[theorem]{Definition/Proposition}

\newtheorem*{claim*}{Claim}
\newtheorem{remark}[theorem]{Remark}

\newtheorem*{remark*}{Remark}
\newtheorem*{remarks*}{Remarks}
\newtheorem*{definition*}{Definition}

\usepackage{calrsfs}

\usepackage[T1]{fontenc}
\usepackage{textcomp}
\usepackage{times}
\usepackage[scaled=0.92]{helvet}

\renewcommand{\tilde}{\widetilde}
\newcommand{\isomto}{\overset{\sim}{\rightarrow}}
\newcommand{\C}{\mathbf{C}}
\newcommand{\Q}{\mathbf{Q}}

\newcommand{\Z}{\mathbf{Z}}
\newcommand{\N}{\mathbf{N}}

\newcommand{\G}{\mathbf{G}}
\newcommand{\U}{\mathbf{U}}
\newcommand{\PP}{\bf{P}}
\newcommand{\B}{\mathbf{B}}
\newcommand{\LL}{\bf{L}}
\newcommand{\K}{\mathbf{K}}

\newcommand{\T}{\mathbf{T}}

\newcommand{\HH}{\mathcal{H}}

\usepackage{amsmath,amssymb,amsfonts,amsthm}

\begin{document}

\title[Hecke algebras for $\mathrm{GL}_n$ over local fields]{Hecke algebras for $\mathrm{GL}_n$ over local fields}
\author[V.\ Karemaker]{Valentijn Karemaker}
\thanks{V.Z.Karemaker@uu.nl. Mathematisch Instituut, Universiteit Utrecht, Postbus 80.010, 3508 TA Utrecht, Nederland}

\begin{abstract}
We study the local Hecke algebra $\mathcal{H}_G(K)$ for $G = \mathrm{GL}_n$ and $K$ a non-archimedean local field of characteristic zero. We show that for $G = \mathrm{GL}_2$ and any two such fields $K$ and $L$, there is a Morita equivalence $\mathcal{H}_G(K) \sim_M \mathcal{H}_G(L)$, by using the Bernstein decomposition of the Hecke algebra and determining the intertwining algebras that yield the Bernstein blocks up to Morita equivalence. By contrast, we prove that for $G = \mathrm{GL}_n$, there is an algebra isomorphism $\mathcal{H}_G(K) \cong \mathcal{H}_G(L)$ which is an isometry for the induced $L^1$-norm if and only if there is a field isomorphism $K \cong L$.
\end{abstract}

\maketitle


\section{Introduction}
The central object of study in this paper is the Hecke algebra for $\mathrm{GL}_n$, over a non-archimedean local field of characteristic zero.
 
\begin{introdefinition}\label{localalgebra}
Let $G = \mathrm{GL}_n$, $n \geq 2$, denote the $n$-dimensional general linear algebraic group, and let $K$ be a non-archimedean local field of characteristic zero. The \emph{(local) Hecke algebra} $\HH_G(K) = C_c^{\infty}(G(K),\C)$ of $G$ over $K$ is the algebra of locally constant compactly supported complex-valued functions on $G(K)$, with the convolution product
\begin{equation}\label{conv}
\Phi_1 \ast \Phi_2 : g \mapsto \int_{G(K)} \Phi_1(gh^{-1})\Phi_2(h)d\mu_{G(K)}(h)
\end{equation}
for $\Phi_1$, $\Phi_2 \in \HH_G(K)$.
\end{introdefinition}

The main question addressed in this paper is: to what extent does (the representation theory of) the local Hecke algebra $\mathcal{H}_G(K)$ determine the field $K$?\\

Hecke algebras are complex algebras with a rich arithmetic structure. In particular, (admissible) representations of the Hecke algebra -- or equivalently, modules over the Hecke algebra -- correspond to (admissible) representations of $\mathrm{GL}_n$. By the Langlands correspondence, which was proven (for $\mathrm{GL}_n$ over $p$-adic fields) by Harris and Taylor \cite{HT} and Henniart \cite{HenGLn}, equivalence classes of admissible representations of  $\mathrm{GL}_n$ in turn are in bijection with equivalence classes of $n$-dimensional Frobenius semisimple representation of the Weil-Deligne group $W'_K$ of $K$, see e.g.\ \cite{wed}. Since $W'_K$ is a group extension of the Weil group $W_K$, from which there exists a continuous homomorphism with dense image in the absolute Galois group $G_K$ of $K$, this places our problem in an anabelian context. It is known that $G_K$ does not determine the field $K$ uniquely \cite{yama}; see Section \ref{anab} for more details.\\

 Our first main result is the following.

\begin{introtheorem}\label{introB}
Let $K$ and $L$ be two non-archimedean local fields of characteristic zero and let $G = \mathrm{GL}_2$. 
Then there is always 
a Morita equivalence $\HH_G(K) \sim_M \HH_G(L)$.
\end{introtheorem}

The Morita equivalence implies that the respective categories of modules over the Hecke algebras of $K$ and $L$ are isomorphic. Equivalently, we find that the representation categories of $\mathcal{H}_G(K)$ and $\mathcal{H}_G(L)$, hence of $\mathrm{GL}_2(K)$ and $\mathrm{GL}_2(L)$, are isomorphic. That is, the module structure of the complex representations of $\mathrm{GL}_2$ over a local field as above does not depend on the local field.

To prove the theorem, we make use of the decomposition of the Hecke algebra into Bernstein blocks. The structure of the blocks is determined up to Morita equivalence, using the representation theory of $p$-adic reductive groups. The preliminaries are collected in Section 2, after which Theorem \ref{introB} is proven in Section 3.

By contrast, returning to $G = \mathrm{GL}_n$ and imposing an analytic condition, we obtain the following theorem, whose proof takes up Section 4.

\begin{introtheorem}\label{introA}
Let $K$ and $L$ be two non-archimedean local fields of characteristic zero and let $G = \mathrm{GL}_n$. 
Then there is an $L^1$-isomorphism of local Hecke algebras $\HH_G(K) \cong \HH_G(L)$ if and only if there is a field isomorphism $K \cong L$.
\end{introtheorem}

Here, an $L^1$-isomorphism of Hecke algebras is an isomorphism which respects the $L^1$-norm that is induced by the Haar measure. The proof first uses the Stone-Weierstrass theorem and density results to reduce to the case of an isomorphism between the group algebras $L^1(G(K)) \cong L^1(G(L))$, and then results due to Wendel \cite{MR0049910} and Kawada \cite{Kawada} to reduce to a group isomorphism $G(K) \cong G(L)$, which implies that $K \cong L$ by classical results on general linear groups. In an earlier paper \cite{CK}, we discussed the adelic analogue of this question. In particular, Theorem \ref{introA} and its proof can be compared with Corollary 6.4 of \cite{CK}. Section \ref{global} discusses a global version of Theorem 2.

Finally, in Section 5, we discuss some open problems.


\section*{Acknowledgements}
I would like to thank Maarten Solleveld for many enlightening discussions, pointing out helpful references, and proofreading of this paper. I am also grateful to Gunther Cornelissen for suggesting this problem, helpful conversations and careful proofreading.


\section{Preliminaries}

In this section, we collect the results from representation theory and on Bernstein decomposition which we need to prove Theorem \ref{introB}.

\subsection*{Representation theory of $\mathrm{GL}_n$}\label{rep}

We will write $\G = G(K)$ from now on, and study and classify representations $\pi \colon \G \to V$ where $V$ is a (possibly infinite-dimensional) complex vector space. More details can be found in e.g.\ \cite{BZ}, \cite{blondel}.

\begin{definition}\label{adm}
The representation $\pi \colon \G \to V$ is called \emph{admissible} if it satisfies the following two conditions:
\begin{enumerate}
\item the stabiliser $\mathrm{Stab}_{\G}(v)$ of any $v \in V$ is an open subgroup of $\G$,
\item for any open subgroup $\G' \subset G(\mathcal{O}_K)$, the space $\{v \in V \colon \pi(g')v = v \textrm{ for all } g' \in \G' \}$ is finite-dimensional.
\end{enumerate}
\end{definition}

\begin{remark} A representation $\pi$ as in Definition \ref{adm} is called \emph{smooth} if it satisfies only the first condition. Clearly, every admissible representations is smooth. Proposition 2 of \cite{blondel} (due to M.-F. Vign{\'e}ras) shows that any smooth irreducible complex representation is admissible. Hence, ``smooth irreducible" and ``admissible irreducible" will be used interchangeably.
\end{remark}

\begin{definition}\label{admhecke}
A representation $\pi' \colon \mathcal{H}_G(K) \to V$ is called \emph{admissible} if it satisfies the following two conditions:
\begin{enumerate}
\item for every $v \in V$, there is an element $f \in \mathcal{H}_G(K)$ such that $\pi(f)v = v$,
\item for every $f \in \mathcal{H}_G(K)$, we have $\mathrm{dim}(\pi(f)V) < \infty$.
\end{enumerate}
\end{definition}

Smooth representations of $\G$ correspond to representations $\pi'$ for which $V$ is a non-degenerate $\mathcal{H}_G(K)$-module \cite{BDKV}. Analogously, admissible representations of $\G$ correspond to admissible representations of $\mathcal{H}_G(K)$ and vice-versa, see e.g.\ $(2.1.13)$ of \cite{wed}.  

\begin{definition}\label{chars}
A \emph{quasicharacter} $\chi$ of $K^{\times}$ is a continuous homomorphism $\chi \colon K^{\times} \to \C^*$. It is called \emph{unramified} if it is trivial on $\mathcal{O}^{\times}_K$. Any unramified quasicharacter is of the form $\vert \cdot \vert^{z}$ for some value of $z$. 
\end{definition}

\begin{lemma}[cf.\ (2.1.18) of \cite{wed}] 
Every irreducible admissible representation $\pi$ which is finite-dimensional is in fact one-dimensional and there exists a quasicharacter $\chi$ such that $\pi(g) = \chi(\det g)$ for all $g \in \G$.
\end{lemma}

Now we turn our attention to the infinite-dimensional representations.

\begin{defprop}\label{L}
A \emph{parabolic subgroup} $\PP$ of $\G$ is such that $\G/\PP$ is complete. Equivalently, $\PP$ contains a Borel subgroup $\B$. Parabolic subgroups are the normalisers of their unipotent radicals, and every $\PP$ is the semidirect product of this unipotent radical and a $K$-closed reductive group $\LL$. This $\LL$ is called the \emph{Levi subgroup} of $\PP$.
\end{defprop}

\begin{remark}
The proper parabolic subgroups of $\mathrm{GL}_n(K)$ are the block upper triangular matrices and their conjugates. For instance, when $n=2$, these are precisely the Borel subgroups, which are $\T \ltimes \U$ with $\T$ a maximal torus and $\U$ a maximal unipotent subgroup. That is, all the Levi subgroups in $\mathrm{GL}_2(K)$ are the maximal tori, i.e., the diagonal $2 \times 2$ matrices. 
\end{remark}

\begin{definition}\label{parind}
Let $\tau$ be a smooth representation of a Levi subgroup $\LL$ of a parabolic subgroup $\PP$ of $\G$. After inflation, we may assume that $\tau$ is a representation of $\PP$. The \emph{parabolic induction} $\mathrm{ind}_{\PP}^{\G}(\tau)$, also denoted $\rho(\tau)$, is the space of locally constant functions $\phi$ on $\G$ which satisfy 
\[
\phi(pg) = \delta_P(p)^{\frac{1}{2}}\tau(p)\phi(g)
\]
for all $g \in \G$ and $p \in \PP$. The normalising factor $\delta_P = \Delta_P^{-1}$ is the inverse of the \emph{modular character} $\Delta_P$ which satisfies $\Delta_P(\mathrm{diag}(a_1,\ldots,a_n)) = \vert a_1 \vert^{1-n}\vert a_2 \vert^{3-n}\ldots \vert a_n \vert^{n-1}$, cf.\ \cite{Prasad}, Ex. 2.6 . 
Parabolic induction preserves smoothness and admissibility but not necessarily irreducibility. 
\end{definition}

\begin{definition}\label{cuspidal}
An infinite-dimensional irreducible admissible representation $\pi \colon \G \to V$ is called \emph{(absolutely) cuspidal} or \emph{supercuspidal} if it is not a subquotient of a representation that is parabolically induced from a \emph{proper} parabolic subgroup of $\G$.
\end{definition}

\begin{definition}\label{segment}
Using the notation of \cite{BZ}, a \emph{partition} $(n_1,\ldots,n_r)$ of $n$ means a partition of $\{1, 2, \ldots, n\}$ into segments $(1,\ldots,n_1),(n_1+1,\ldots,n_1+n_2), \ldots, (n_1+n_2+\ldots+n_{r-1}+1, \ldots, n)$ of respective lengths $n_i$. We will write $(n_1,\ldots,n_r) \perp n$ for such a partition. 

For any $n_i$ appearing in a partition of $n$, write $\Delta_i = \{ \sigma_i, \sigma_i\vert\cdot\vert,\ldots,\sigma_i\vert\cdot\vert^{n_i-1} \}$ for $i = 1, \ldots ,r$ and $\sigma_i$ an irreducible supercuspidal representation of $\mathrm{GL}_{n_i}(K)$. The $\Delta_i$ are also called segments, and we say that $\Delta_i$ \emph{precedes} $\Delta_j$ if $\Delta_i \not\subset \Delta_j$ and $\Delta_j \not\subset \Delta_i$, if $\Delta_i \cup \Delta_j$ is also a segment, and $\sigma_i = \sigma_j \vert\cdot\vert^k$ for some $k > 0$. 
\end{definition}

Now compare Definition \ref{cuspidal} with the following result (cf.\ Theorem 6.1 of \cite{zel}, Corollary 3.27 of \cite{BZ} and pp.\ 189-190 of \cite{Prasad}).


\begin{lemma}\label{lang}
For any partition $(n_1,\ldots,n_r)$ of $n$ and a choice of segments so that $\Delta_i$ and $\Delta_{i+1}$ ($i=1,\ldots, r$) do not precede each other, there exists a corresponding induced representation, denoted $\mathrm{ind}_{\PP}^{\G}(\sigma_1 \ldots \otimes \sigma_r)$, whose unique irreducible quotient is an irreducible admissible representation of $\G$. Any irreducible admissible representation of $\G$ is equivalent to such a quotient representation. \end{lemma}

Hence, supercuspidal representations can be viewed as the building blocks of admissible representations of $\G$. 
This concludes the classification of admissible representations of $\G$.

\begin{remark}\label{classificationGl2}
Let now $n=2$, so that $G = \mathrm{GL}_2$ and $\G = \mathrm{GL}_2(K)$. Every infinite-dimensional irreducible admissible representation $\pi$ which is not supercuspidal is then contained in $\rho(\mu_1,\mu_2)$ for some quasicharacters $\mu_1$, $\mu_2$ of $K$.
If $\mu_1\mu_2^{-1} \neq |\cdot |^{\pm 1}$ then $\rho(\mu_1,\mu_2)$ and $\rho(\mu_2,\mu_1)$ are equivalent and irreducible. We call a representation of this kind a \emph{(non-special) principal series representation}. 

If $\rho(\mu_1,\mu_2)$ is reducible, it has a unique finite-dimensional constituent, and a unique infinite-dimensional constituent $\mathcal{B}_s(\mu_1,\mu_2)$, also called a \emph{special representation}. 
For special representations, there exists a quasicharacter $\chi$ such that $\mu_1 = \chi |\cdot |^{-\frac{1}{2}}$ and $\mu_2 = \chi |\cdot |^{\frac{1}{2}}$. Moreover, all special representations are twists of the so-called \emph{Steinberg representation} $\mathrm{St}_{\G}$ of $\G$ by quasicharacters $\chi \circ \det$.

Summarising, any irreducible admissible representation $\pi \colon \G \to V$ satisfies one of the following:
\begin{itemize}
\item[(1):] it is absolutely cuspidal;
\item[(2a):] it is a principal series representation $\pi(\mu_1,\mu_2)$ for some quasicharacters $\mu_1, \mu_2$;
\item[(2b):] it is a special representation $\sigma(\mu_1,\mu_2)$ for some quasicharacters $\mu_1, \mu_2$;
\item[(3):] it is finite-dimensional and of the form $\pi = \chi \circ \det$ for some quasicharacter $\chi$.
\end{itemize}
More details on $\mathrm{GL}_2$ can be found in e.g.\ \cite{JL}, \cite{bushen}.
\end{remark}

\subsection*{Bernstein decomposition}

We will introduce the Bernstein decomposition, using \cite{bushkutz} and \cite{BDKV} as our main references. Let $\G = \mathrm{GL}_n(K)$ as before.

\begin{definition}\label{inertialclass}
Let $\LL$ be a Levi subgroup of some parabolic $\PP$ inside $\G$ and let $\sigma$ be an irreducible cuspidal representation of $\LL$. We define the \emph{inertial class} $[\LL,\sigma]_{\LL}$ of $(\LL,\sigma)$ in $\LL$ to be all the cuspidal representations $\sigma'$ of $\LL$ such that $\sigma \cong \sigma'\chi$ for $\chi$ an unramified character of $\LL$. Let $\mathcal{B}(\G)$ be the set of all inertial equivalence classes in $\G$.
\end{definition}

We need the following refinement of Lemma \ref{lang}.

\begin{theorem}\label{support}
For every smooth irreducible representation $(\pi,V)$ of $\G$ there exists a parabolic $\PP$ in $\G$ with Levi subgroup $\LL$, and an irreducible supercuspidal representation $\sigma$ of $L$, such that $(\pi, V)$ is equivalent to a subquotient of the parabolic induction $\mathrm{Ind}_{\PP}^{\G} (\sigma)$ 
\textup{\cite{jac}}. The pair $(\LL,\sigma)$ is determined up to conjugacy; the corresponding intertial class $s = [\LL,\sigma]_{\G}$ is unique \textup{\cite{cas}}.
\end{theorem} 

\begin{definition}
The pair $(\LL,\sigma)$ in the previous theorem is called the \emph{cuspidal support} of $(\pi,V)$; the corresponding intertial class $s = [\LL,\sigma]_{\G}$ is called the \emph{inertial support} of $(\pi,V)$.
\end{definition}

\begin{lemma}[Prop. 2.10 of \cite{BDKV}]
Denote by $\mathfrak{R}(\G)$ the category of smooth representations $(\pi,V)$ of $\G$ and by $\mathfrak{R}^s(\G)$ the full subcategory, whose objects are such that the inertial support of all their respective irreducible $\G$-subquotients is $s$. Then there is a direct product decomposition of categories
\[
\mathfrak{R}(\G) = \prod_{s \in \mathcal{B}(\G)} \mathfrak{R}^s(\G).
\]
\end{lemma}


\begin{corollary}\label{Bernsteinblock}
Let $\mathcal{H}^s_{G}(K)$ be the two-sided ideal of $\mathcal{H}_{G}(K)$ corresponding to all smooth representations $(\pi,V)$ of $\G$ of inertial support $s = [\LL,\sigma]_{\G}$. That is, $\mathcal{H}^s_{G}(K)$ is the unique and maximal $\G$-subspace of $\mathcal{H}_{G}(K)$ lying in $\mathfrak{R}^s(\G)$. We call $\mathcal{H}^s_{G}(K)$ a \emph{Bernstein block}.
\end{corollary}

\begin{definition}
The Hecke algebra $\mathcal{H}_{G}(K)$ has a \emph{Bernstein decomposition}
\[
\mathcal{H}_{G}(K) = \bigoplus_{s \in \mathcal{B}(\G)} \mathcal{H}^{s}_{G}(K).
\]
\end{definition}

\begin{definition}\label{intertw}
Let $(\rho,W)$ be a smooth representation of a compact open subgroup $\K$ of $\G$, whose contragredient representation is denoted $(\check{\rho}, \check{W})$.\\
The \emph{$\rho$-spherical Hecke algebra} $\mathcal{H}(\G,\rho)$ is the unital associative $\C$-algebra of finite type, consisting of compactly supported functions $f \colon \G \to \mathrm{End}_{\C}(\check{W})$ satisfying $f(k_1 g k_2) = \check{\rho}(k_1) f(g) \check{\rho}(k_2)$ for all $g \in \G, k_1,k_2 \in \K$. It is also called the \emph{intertwining algebra}, since 
\[
\mathcal{H}(\G,\rho) \cong \mathrm{End}_{\G}(\mathrm{ind}_{\K}^{\G}(\rho))
\]
by (2.6) of \cite{bushkutz}, where $\mathrm{ind}$ denotes compact induction.
\end{definition}

\begin{proposition}[Prop. 5.6 of \cite{bushkutz}]\label{block}
Every $\mathcal{H}^{s}_{G}(K)$ is a non-commutative, non-unital, non-finitely generated non-reduced $\C$-algebra, which is Morita equivalent to some intertwining algebra $\mathcal{H}(\G,\rho)$.
\end{proposition}

\begin{proof}[Sketch of proof]
For every equivalence class $s$ there exist a compact open subgroup $\K$ of $\G$, a smooth representation $(\rho,W)$ of $\K$ and an idempotent element $e_{\rho} \in \mathcal{H}_{G}(K)$ (cf.\ (2.9) of \cite{bushkutz}) which satisfies
\[
e_{\rho}(x) = \begin{cases}
   \frac{\mathrm{dim}(\rho)}{\mu_{\G}(\K)} \mathrm{tr}_{W}(\rho(x^{-1})) & \text{if } x \in \K \\
   0       & \text{if } x \in \G, x \not\in \K
  \end{cases},
\]
such that 
\[
\mathcal{H}^s_{G}(K) = \mathcal{H}_{G}(K) \ast e_{\rho} \ast \mathcal{H}_{G}(K).
\]
There is a Morita equivalence (cf. \cite{conmar}, Lemma 2)
\[
\mathcal{H}_{G}(K) \ast e_{\rho} \ast \mathcal{H}_{G}(K) \sim_M e_{\rho} \ast \mathcal{H}_{G}(K) \ast e_{\rho}
\]
and the latter is proven in (2.12) of \cite{bushkutz} to be isomorphic as a unital $\C$-algebra to
\begin{equation}\label{eq1}
e_{\rho} \ast \mathcal{H}_{G}(K) \ast e_{\rho} \cong \mathcal{H}(\G,\rho) \otimes_{\C} \mathrm{End}_{\C}(W)
\end{equation}
where $\mathcal{H}(\G,\rho)$ is as in Definition \ref{intertw}. In particular, there is a Morita equivalence
\begin{equation}\label{eq2}
\mathcal{H}^s_{G}(K) \sim_M \mathcal{H}(\G,\rho),
\end{equation}
i.e., the categories of modules over the left resp. right hand side of (\ref{eq2}) are equivalent. 
\end{proof}

\section{Proof of Theorem \ref{introB}}

\begin{definition}
The \emph{(extended) affine Weyl group} of $\mathrm{GL}_n$ is $\tilde{W}_n \cong S_n \ltimes \Z^n$, where the symmetric group $S_n$ acts by permuting the factors of $\Z^n$. We denote its group algebra by 
\[
\C[\tilde{W}_n] = \C[S_n \ltimes \Z^n].
\]
\end{definition}

\noindent In this section, we will prove the following result, which immediately implies Theorem \ref{introB}.

\begin{theorem}\label{B}
Let $K$ be a non-archimedean local field of characteristic zero and $G = \mathrm{GL}_2$. Then up to Morita equivalence, the Bernstein decomposition of $\mathcal{H}_G(K)$ is always of the form
\begin{equation}\label{con}
\mathcal{H}_{\mathrm{GL}_2}(K) \sim_M \bigoplus_{\N} \left( \C[T,T^{-1}] \oplus \C[X,X^{-1},Y,Y^{-1}] \oplus \frac{\C[S,T,T^{-1}]}{\langle S^2-1, T^2S-ST^2 \rangle} \right).
\end{equation}
In particular, if $K$ and $L$ are any two non-archimedean local fields of characteristic zero, then 
\[
\HH_G(K) \sim_M \HH_G(L).
\]
\end{theorem}

\begin{proof}
By Lemma \ref{lang}, every irreducible representation of $\G$ is a subquotient of a parabolically induced representation $\mathrm{ind}_{\PP}^{\G}(\sigma_1 \ldots \otimes \sigma_r)$, where the $\sigma_i$ are irreducible supercuspidal representations of $\mathrm{GL}_{n_i}$ and $(n_1,\ldots,n_r)$ is a partition of $n$, so that consecutive segments do not precede each other. Note that $n_i$ is the multiplicity of $\sigma_i$ in the tensor product, so that $n_i = 1$ or $2$ always.

Proposition \ref{block} implies that to determine the corresponding Bernstein blocks $\mathcal{H}_{G}^s(K)$ of the Hecke algebra up to Morita equivalence, it suffices to determine all intertwining algebras $\mathcal{H}(\G,\rho)$ that occur.
To do this, we need the following definition.


\begin{definition}[cf.\ (5.4.6) of \cite{bushkutz0}]\label{IH}
Let $m \in \Z_{>0}$ and $r \in \C^{\times}$. The \emph{affine Hecke algebra} $\mathcal{H}(m,r)$ is the associative unital $\C$-algebra generated by elements $S_i$ ($1 \leq i \leq m-1$), $T$, $T^{-1}$, satisfying the following relations:
\begin{enumerate}
\item $(S_i + 1)(S_i - r) = 0$ for $1 \leq i \leq m-1$,
\item $T^2S_1 = S_{m-1}T^2$,
\item $TS_i = S_{i-1}T$ for $2 \leq i \leq m-1$,
\item $S_i S_{i+1} S_i = S_{i+1}S_iS_{i+1}$ for $1 \leq i \leq m-2$,
\item $S_i S_j = S_j S_i$ for $1 \leq i,j \leq m-1$ such that $\vert i-j \vert \geq 2$.
\end{enumerate}
Note that when $m=1$, we have $\mathcal{H}(1,r) \cong \C[T,T^{-1}]$ for any value of $r$. Moreover, note that when $m \leq 2$, relations $(3)$,$(4)$ and $(5)$ are vacuous.
\end{definition}

By the Main Theorem of \cite{bushkutz2}, the intertwining algebra corresponding to $\mathrm{ind}_{\PP}^{\G}(\sigma_1\otimes~\ldots~\otimes~\sigma_r)$ is isomorphic to the tensor product $\otimes_{i=1}^r \mathcal{H}(n_i,q^{k_i})$ of affine Hecke algebras, where $n_i \leq 2$ since $n = 2$. Here, $q$ is the size of the residue field of $K$, while $k_i$ is the so-called torsion number of $\sigma_i$, cf.\ \cite{conmar}, p.22. In particular, $q^{k_i} \neq -1$ always. A priori, the Hecke algebra $\mathcal{H}(2,q^{k_i})$ depends on $q^{k_i}$. However, we now prove the following.

\begin{lemma}
For any $r \neq -1$, there is an algebra isomorphism $\mathcal{H}(2,r) \cong \C[\tilde{W}_2]$.
\end{lemma}
\begin{proof}
First let $r=1$. Let $\varpi$ be a uniformiser of $K$. Since $\varpi$ is not a root of unity, we may alternatively (cf.\ \cite{bushkutz0}, pp.\ 177--178) write $\tilde{W}_2 = \langle\Pi\rangle\ltimes W$, where 
\[
\Pi = \left(
\begin{array}{cc}
0 & 1 \\
\varpi & 0
\end{array}
\right),
\]
and $W$ is generated by 
\[
s_1 = \left(
\begin{array}{cc}
0 & 1 \\
1 & 0
\end{array}
\right).
\]
We may check that $s_1$ has order 2 and that sending $S_1 \mapsto s_1$, and $T \mapsto \Pi$ (and $T^{-1} \mapsto \Pi^{-1}$) yields an algebra isomorphism $\mathcal{H}(2,1) \to \C[\tilde{W}_2]$.\\

Now let $r \in \C^{\times}\setminus \{-1\}$ and let (cf.\ p.\ 113 of \cite{xi})
\[
\bar{s}_1 = \left(\frac{q+1}{2}s_1 + \frac{q-1}{2}\right) \in \C[\tilde{W}_2].
\]
Then relation $(2)$,
\[
\Pi^2\bar{s}_1 = \bar{s}_{1}\Pi^2
\]
still holds. Hence, the map $S_1 \mapsto \bar{s}_1$, and $T \mapsto \Pi$ (and $T^{-1} \mapsto \Pi^{-1}$) determines an algebra isomorphism $\mathcal{H}(2,q) \to \C[\tilde{W}_2]$, for any $r$ other than $r= -1$.
\end{proof}

It follows from the lemma that the intertwining algebra for the partition $(n_1, \ldots, n_r)$ of $n$, corresponding to the representation $\mathrm{ind}_{\PP}^{\G}(\sigma_1\otimes~\ldots~\otimes~\sigma_r)$, is isomorphic to the $\C$-algebra $\otimes_{i=1}^{r} \C[\tilde{W}_{n_i}]$.\\

Finally, we show that any such algebra $\otimes_{i=1}^{r} \C[\tilde{W}_{n_i}]$ occurs countably infinitely many times in the Bernstein decomposition. For this, we use the classification of Remark \ref{classificationGl2}. The reader may compare this to the explicit description of the intertwining algebras in \cite{secherre}, Example 3.13.

\begin{itemize}
\item[(1):] A \emph{supercuspidal representation} $(\pi,V)$ corresponds to an inertial class $s = [\G,\rho]_{\G}$ where $\rho$ is itself an irreducible supercuspidal representation. The corresponding intertwining algebra is $\mathcal{H}(\G,\rho) \cong \mathcal{H}(1,q) \cong \C[T,T^{-1}]$, for $q$ some power of $p$. The uncountably infinitely many equivalence classes of supercuspidal representations are indexed by characters of quadratic extensions of $K$.
\item[(2a):] The principal series representations are constituents of representations of the form $\mathrm{ind}_{\B}^{\G}(\chi_1,\chi_2)$ for a choice of Borel subgroup $\B$ of $\G$ and characters $\chi_1$ and $\chi_2$. Therefore, up to inertial equivalence, we find $\rho(\chi_1\vert\cdot\vert^{z},\chi_2\vert\cdot\vert^{z'}) = \mathrm{ind}_{B}^{G}(\chi_1\vert\cdot\vert^{z},\chi_2\vert\cdot\vert^{z'})$ for some characters $\chi_1$ and $\chi_2$, and some values $z$, $z'$. 

\emph{Non-special representations} then correspond to a choice of $\chi_1$, $\chi_2$ such that $\chi_1 \chi_2^{-1} \neq \vert\cdot\vert^{\pm 1}$ (i.e. $\chi_1$ and $\chi_2$ are not inertially equivalent), or a choice of $\chi$, $z$, $z'$ such that $\vert z - z' \vert \neq 1$. The corresponding inertial class is $s = [\T,\rho]_{\G}$, where $\T$ is a maximal torus in $\B$. For such $\rho$, we have $\mathcal{H}(\G,\rho) \cong \mathcal{H}(1,q) \otimes \mathcal{H}(1,q') \cong \C[X,X^{-1},Y,Y^{-1}]$, for $q$ and $q'$ some powers of $p$. 

We also see that these representations are indexed by the characters of $(\mathcal{O}_K^{\times})^2$ modulo the action of $S_2$, which is a countably infinite group.
\item[(2b/3):] A \emph{special representation} is the infinite-dimensional irreducible subquotient $\mathrm{St}_{\G}\chi\vert\cdot\vert^{z+\frac{1}{2}}$ of the reducible representation $\rho = \rho(\chi\vert\cdot\vert^{z+1},\chi\vert\cdot\vert^{z})$ for some $\chi$ and $z$, and corresponds to $s = [\T,\rho]_{\G}$. The \emph{finite-dimensional representations} appear as the finite-dimensional irreducible subquotients of the same $\rho$.

Hence, the corresponding inertial equivalence classes $s$ are indexed by the character group of $\mathcal{O}_K^{\times}$, which is countably infinite. The corresponding intertwining algebras for both special and finite-dimensional representations are 
\[
\mathcal{H}(2,q) \cong \C[\tilde{W}_2] \cong \C[S,T,T^{-1}]/\langle S^2-1, T^2S-ST^2 \rangle.
\]
\end{itemize}

This finishes the proof of Theorem \ref{B} and hence of Theorem \ref{introB}.
\end{proof}

\section{$L^1$-isomorphisms of local Hecke algebras}

Theorem \ref{B} shows that local Hecke algebras for $\mathrm{GL}_2$ are always Morita equivalent. By contrast, Theorem \ref{A} below implies that the $L^1$-isomorphism type of a local Hecke algebra for any $\mathrm{GL}_n$ ($n \geq 2$) determines the local field up to isomorphism.

\begin{definition}\label{pointgroup}
Let $G = \mathrm{GL}_n$ and let $K$ again be a non-archimedean local field of characteristic zero. Since $K$ is locally compact, $G(K)$ is then a locally compact topological group, whose topology is induced by the topology of $K$. Its group structure is induced by that of $G$. Moreover, it is equipped with a (left) invariant Haar measure $\mu_{G(K)}$ which satisfies $\mu_{G(K)}(G(\mathcal{O}_K)) = 1$. The \emph{group algebra} $L^1(G(K))$ of $G$ over $K$ is the algebra of complex-valued $L^1$-functions with respect to $\mu_{G(K)}$, under the convolution product as in Equation (\ref{conv}).
\end{definition}

\begin{definition}
Let $K$ and $L$ both be non-archimedean local fields. An isomorphism of Hecke algebras $\Psi \colon \HH_G(K) \isomto \HH_G(L)$ which is an isometry for the $L^1$-norms arising from the Haar measures (i.e., which satisfies $||\Psi(f)||_1=||f||_1$ for all $f \in \HH_G(K)$) is called an \emph{$L^1$-isomorphism.}
\end{definition}

\begin{theorem}\label{A}
Let $K$ and $L$ be two non-archimedean local fields of chracteristic zero.
Then there is an $L^1$-isomorphism of local Hecke algebras $\HH_G(K) \cong \HH_G(L)$ if and only if there is a field isomorphism $K \cong L$.
\end{theorem}

\begin{remark} 
In the statement of Theorem \ref{A}, the field isomorphism $K \cong L$ is automatically a topological field isomorphism: an abstract field isomorphism $K \cong L$ will restrict to the multiplicative groups: $K^{\times} \cong L^{\times}$. However, $K^{\times}$ will have an infinite divisible $p$-subgroup if and only if $K$ is an extension of $\Q_p$, by Corollary 53.4 of \cite{classicalfields}. Thus, an abstract isomorphism determines the residue characteristic of $K$ and $L$ uniquely, so they are finite extensions of the same $\Q_p$. The valuation of $\Q_p$ extends uniquely to both $K$ and $L$, so that in particular the valuations on $K$ and $L$ will be equivalent and hence generate the same topology. 
\end{remark}

\begin{proof}
Firstly, we claim that an $L^1$-isomorphism $\HH_G(K) \cong \HH_G(L)$ implies an $L^1$-isomorphism of algebras $L^1(G(K)) \cong L^1(G(L))$. 

Because $G(K)$ is a Hausdorff space, $\mathcal{H}_G(K)$ is point separating. Moreover, the Hecke algebra vanishes nowhere, since it contains the characteristic function of any compact $\mathbf{K} \subseteq G(K)$. By $7.37.\textup{b}$ of \cite{HS}, we can therefore apply the Stone-Weierstrass theorem for locally compact spaces, to conclude that $\HH_G(K)$ is dense inside the algebra $C_0(G,K)$ of functions $f \colon G(K) \to \C$ which vanish at infinity (meaning that $|f(x)|< \varepsilon$ outside a compact subset of $G(K)$), under the sup-norm.

A fortiori, $\HH_G(K)$ is dense, again under the sup-norm, inside the algebra $C_c(G,K)$ of compactly supported functions $G(K) \to \C$, and hence it is also dense under the $L^1$-norm. Now $C_c(G,K)$ is dense in $L^1(G(K))$, proving the claim.

Secondly, by results due to Wendel \cite{MR0049910} and Kawada \cite{Kawada}, an $L^1$-isometry of group algebras of locally compact topological groups is always induced by an isomorphism of the topological groups. Therefore, an $L^1$-isometry $L^1(G(K)) \cong L^1(G(L))$ implies a group isomorphism $G(K) \cong G(L)$.

Finally, the fact that $G(K) \cong G(L)$ implies that $K \cong L$ is a classical result, cf.\ Theorem 5.6.10 of \cite{Omeara}.
\end{proof}

\begin{remark}
This theorem can be viewed as a local and complex version of Theorem 6.3 ("Theorem E") of \cite{CK}, which deals with the adelic and real analogue (but also holds over $\C$). The proof above is analogous to that of Theorem 6.3. In the last step, instead of citing the literature, we could use Theorem G of \cite{CK}, since the local version of this also implies a field isomorphism $K \cong L$.
\end{remark}

\section{Discussion}

The results in this paper naturally inspire some further questions.

\subsection{Generalisations of Theorem \ref{B}}
\begin{enumerate}
\item We have seen that for $\mathrm{GL}_2$, up to Morita equivalence, $\mathcal{H}_G(K)$ does not depend on $K$. Does the same hold up to algebra isomorphism?
\item An extension of the proof of Theorem \ref{B} to $\mathrm{GL}_n$, $n > 2$, is obstructed by the braid relations ($(4)$ of Definition \ref{IH}) among the generators of the affine Heceke algebras. This is pointed out by Xi in (11.7) of \cite{xi}, where he proves that $\mathcal{H}(3,q) \not\equiv \C[\tilde{W}_3]$ for $q \neq 1$. In fact, Yan proves in \cite{yan} that any two affine Hecke algebras $\mathcal{H}(n,q)$ and $\mathcal{H}(n,q')$ of type $\tilde{A}_2$ are not Morita equivalent when $q \neq q'$.
\item One may still ask whether Theorem \ref{B} also holds for other reductive groups $G$ over $K$. It is known that the Hecke algebras of such groups also admit a Bernstein decomposition. However, it is in general much harder to determine the complex algebras that occur as intertwining algebras and to show that these are independent of $K$. We would also want to have a similar classification of the representation theory of such $G$.
\end{enumerate}

\subsection{A global version of Theorem \ref{B}}\label{global}
In the proof of Theorem \ref{B}, we have seen that the residual characteristic $p$ of $K$ does not play a special role. Hence, 
If $K$ is a number field and $G = \mathrm{GL}_n$, $n \geq 2$, we may consider the (adelic) Hecke algebra $\mathcal{H}_G(K)$ as a restricted tensor product of local Hecke algebras $\mathcal{H}_G(K_v)$, with respect to the maximal open compact subgroups $G(\mathcal{O}_v)$:
\[
\mathcal{H}_G(K) = \otimes_{v} \mathcal{H}_{G}(K_v), 
\]
cf. Chapter 9 of \cite{JL} for $G = \mathrm{GL}_2$ and p.320 of \cite{bump0} for $G = \mathrm{GL}_n$. We know that the module category of any $\mathcal{H}_{\mathrm{GL}_2}(K_v)$ is independent of $K_v$ (so in particular independent of the residual characteristic of $K_v$). Hence, a natural question would be to ask whether the module category of $\mathcal{H}_{\mathrm{GL}_2}(K)$ is also independent of $K$. We expect however that the restricted tensor product construction, through the rings of integers $\mathcal{O}_v$, \emph{does} depend on $K$.

\subsection{Local anabelian questions}\label{anab}
Our main goal was to investigate to what extent $\mathcal{H}_G(K)$ for $G = \mathrm{GL}_n$ determines the field $K$.
By the philosophy of the Langlands program, our question roughly translates to asking which representations of the absolute Galois group $G_K$ determine $K$. It therefore fits in a local anabelian context. 

Neukirch and Uchida proved that the absolute Galois group of a number field determines the number field. As Yamagata points out in \cite{yama}, the analogous statement for non-archimedean local fields of characteristic zero is false. However, Jarden and Ritter \cite{jarrit} prove that in this case, the absolute Galois group $G_K$ determines the absolute field degree $[K \colon \Q_p]$ and the maximal abelian subextension of $K$ over $\Q_p$. In addition, Mochizuki \cite{Mochi} and Abrashkin \cite{abrashkin} prove that the absolute Galois group together with its ramification filtration \emph{does} determine a local field of characteristic $0$ resp. $p>0$. We may therefore ask exactly which field invariants of $K$ are determined by $\mathcal{H}_G(K)$.


\subsection{The $L^1$-isomorphism condition in Theorem \ref{A}}
The condition that the isomorphism $\mathcal{H}_G(K) \cong \mathcal{H}_G(L)$ is an isometry for the $L^1$-norm is one which we would like to understand from a categorial viewpoint. Does the $L^1$-isomorphism type of (modules over) a Hecke algebra impose analytic conditions on (certain classes of) the automorphic representations? Or, more in the spirit of Section \ref{anab}, can we relate the $L^1$-isomorphism type of a Hecke algebra $\mathcal{H}_G(K)$ to the ramification filtration of the absolute Galois group $G_K$?

\newpage

\bibliographystyle{amsplain}

\end{document}